\documentclass[12pt]{amsart} 
\usepackage{verbatim}
\usepackage{amsmath,amsthm,amssymb,amsfonts,mathrsfs,color}
\usepackage{latexsym,esint}
\theoremstyle{plain}
\begingroup
\newtheorem{thm}{Theorem}[section]
\newtheorem{lem}[thm]{Lemma}
\newtheorem{prop}[thm]{Proposition}
\newtheorem{defin}[thm]{Definition}
\newtheorem{cor}[thm]{Corollary}
\newtheorem{conj}[thm]{Conjecture}
\endgroup

\theoremstyle{definition}
\begingroup
\newtheorem{rem}[thm]{Remark}
\endgroup

\numberwithin{equation}{section}

\makeatletter
\def\rightharpoonupfill@{\arrowfill@\relbar\relbar\rightharpoonup}
\newcommand{\xrightharpoonup}[2][]{\ext@arrow 0359\rightharpoonupfill@{#1}{#2}}
\makeatother

\newcommand{\N}{\mathbb N} 
\newcommand{\R}{\mathbb R}

\newcommand{\diam}{{\rm diam}}
\renewcommand{\div}{{\rm div}}

\newcommand{\Hh}{{\mathcal H}}

\begin{document}
 
\title{A rigidity result for global Mumford-Shah minimizers in dimension three}
\author[A. Lemenant]{Antoine Lemenant}




\address[A. Lemenant]{Universit\'e Paris Diderot -- Paris 7, CNRS, UMR 7598 Laboratoire Jacques-Louis Lions, Paris, F-75005, France}
\email{lemenant@ljll.univ-paris-diderot.fr}

\date{\today}

\begin{abstract}\vspace{1cm} We study global Mumford-Shah minimizers  in $\R^N$, introduced by Bonnet as blow-up limits of Mumford-Shah minimizers. We prove a new monotonicity formula for the energy of $u$ when the singular set $K$ is contained in a smooth enough cone. We then use this monotonicity to prove that for any reduced global minimizer $(u,K)$ in $\R^3$, if $K$ is contained in a half-plane and touching its edge, then it is the half-plane itself. This  partially answers   to a question of Guy David. 
\end{abstract}
\maketitle

\vspace{2cm}
\tableofcontents
\newpage
\section{Introduction}

Global Mumford-Shah minimizers  was introduced first by Bonnet  \cite{b} in dimension 2  as blow-up limit of minimizers of the Mumford-Shah functional, and are crucial in the regularity result of \cite{b}. They are couples $(u,K)$, with $K\subset \R^N$ a closed set and $u\in \bigcap_{R>0}W^{1,2}(B(0,R)\setminus K)$, that locally minimizes in $\R^N$ the Mumford-Shah energy. More precisely,  for any ball $B\subset\R^N$, $(u,K)$ minimizes the sum  
$$\int_{B\setminus K}|\nabla u|^2dx + \mathcal{H}^{N-1}(K\cap B),$$
 among all compettiting couples $(u',K')$ that are equal to $(u,K)$ outside $B$, and that keep the topology of $K$ (see Definition \ref{defms} below for more details). The $k$-dimensional Hausdorff measure $\mathcal{H}^k$ generalizes the natural notion of $k$-area and will be precisely defined in Section \ref{secprel}.

The list of global minimizers is nearly complete in dimension 2: if $K$ is connected then it is either a line, a half-line, or three half-lines meeting at the origin by 120$^o$. Knowing this without assuming ``connected'' would solve the famous Mumford-Shah conjecture. Some improvements in this direction can be found in \cite{l} and \cite{dl}.

But on the other hand, almost nothing is known about global minimizers in dimension $N>2$. A first discussion is given in Chapter H.76.  of Guy David's book \cite{d},  and a few more results can be found in the paper \cite{lmg}.

For instance it is known in dimension 3 that, if $u$ is locally constant, then $K$ must be a minimal cone (i.e. a Plane, three half planes meeting by $120^o$, or a cone over the regular tetrahedron). This fact looks very natural but is surprisingly not obvious at all, and a rigorous proof is contained in \cite{dreg}. It is also known that a half-plane is a minimizer with the corresponding function $u=cracktip\times \R$, where the cracktip function is $\sqrt{2r/\pi}\sin(\theta/2)$ in polar coordinates. These are the only known examples of global minimizers in $\R^3$, and we refer to \cite{lmg} and \cite{d} for more details.

There is also counter-examples. For instance in \cite{lmg} it is proved that an angular sector cannot be a global minimizer, unless it is a half-plane or a plane. In \cite{merlet} a tentative construction for an extra global minimizer of particular type is proved to fail.

Notice that adding to $K$ a piece of set negligible for $\mathcal{H}^{N-1}$  does not change the fact of being a Mumford-Shah minimizer. Therefore, it is convenient to work with \emph{reduced} minimizers only, for which one cannot remove from $K$ more pieces of set of zero measure without being enable to extend the associated function $u$ as a $W^{1,2}$ function.   For instance, reduced minimizers enjoy nice properties as Ahlfors-regularity, or local $C^1$ regularity at points where the energy of $u$ is small enough, that non-reduced minimizers  can obviously not expect to have in general. As proved in \cite{d}, it is always possible to replace a minimizer by a reduced one (there is even a debate between 3 different possibilities, the one used in this paper is called ``topological reduction'' in \cite{d}), but anyway in the case when the  the global minimizer is coming from a blow-up procedure, then it is automatically reduced \cite[Proposition 40.9]{d}. We refer to \cite[Section 8]{d} for more details about reduced minimizers.

In this paper we focus our attention on the following conjecture which is stated in  \cite[Page 571]{d}.

\begin{conj}[Guy David] \label{conj1} Let $(u,K)$ be a reduced global minimizer in $\R^3$ such that $K$ is  contained in a plane and not empty. Then $K$ must be either a plane  or a half-plane.
\end{conj}

In this paper we give a partial answer to this conjecture. We prove  the following.

\begin{thm} \label{mainth} Let $(u,K)$ be a reduced global minimizer in $\R^3$ such that $K$ is  contained in a half-plane, which contains a point of $K$ on its edge. Then  $K$ must be the whole   half-plane.
\end{thm}

The major obstacle which prevents us to prove the full conjecture will be explain a the end of the introduction. It comes from geometrical requirements to control the first eigenvalue of the Neumann-Laplace-Beltrami operator  on spherical domains in order to get a monotonicity formula for the normalized energy, with the right power. We  are able to control the first eigenvalue of $\mathbb{S}^{N-1}\setminus K$ when $K$ is contained in a half-plane, but not larger.

We stress that in dimension $N=2$, similar results have been obtained  by J-C L\'eger \cite{l}. For instance Proposition 25 in \cite{l} says that  when $N=2$, if $K$ is contained in a line then $K$ is either the whole line, a half-line, or the empty set.  The proof of L\'eger  relies strongly on complex analysis. In particular the beautifull (and planar) formula showing the link between the set $K$ and the complex derivative of $u$ is one of the key ingredient in \cite{l}.  The second main ingredient is the monotonicity formula of Bonnet \cite{b}.

Therefore,  the challenge in Conjecture \ref{conj1} is to find a new strategy in higher dimension, without complex analysis, and prioritary try to find a monotonicity formula like the one of Bonnet valid in higher dimensions. The latter being also a question in page 572 of \cite{d}.

In the present paper we indeed propose a monotonicity formula  to replace the one of Bonnet. It relies on  the famous argument from Alt-Caffarelli-Friedman, and is obtained here for solutions of the Neumann problem in the complement of a set $K\subseteq \R^N$ contained in a cone. According to us, this monotonicity result  is probably as interesting as our main result itself since, as far as we know, it is the first time that some monotonicity is shown for the Neumann problem in dimension $N>2$. On the other hand to be valid, the singular set $K$ must be contained in a cone, which could be slightly restrictive but turns out to be exactly the case in Conjecture \ref{conj1}.

Let  us  recall the monotonicity formula of Bonnet \cite{b} which says that if $K\subset \R^2$ is a closed connected set and $u$ is a local  energy minimizer in $\R^2\setminus K$ (i.e. $u$ is harmonic with Neumann boundary conditions on $K$), then
$$r\mapsto \frac{1}{r}\int_{B_r\setminus K}|\nabla u|^2 dx$$ 
is nondecreasing (see also \cite[Section E.47]{d} or \cite{cl}). The proof of Bonnet seems not to be directly extendable in higher dimensions.

On the other hand, using a classical argument from Alt-Caffarelli-Friedmann \cite{acf}, it is easy to show that, for the Dirichlet problem, if the first eigenvalue of the Dirichlet Laplacian is uniformly bounded from below by $(2N-3)/4$ on $\partial B_r \setminus K$ for a.e. $r>0$, then
$$r\mapsto \frac{1}{r}\int_{B_r\setminus K}\frac{|\nabla u|^2}{|x|^{N-2}} dx$$
is nondecreasing. See for instance  \cite[Lemma 16]{lms} for a proof. Notice that for $N=2$ we recover exactly Bonnet's energy.

Both in  Bonnet \cite{b}  and Alt-Caffarelli-Friedman \cite{acf}, the monotonicity formula is established through an integration by parts which transforms the energy $f(r)=\int_{B_r}|\nabla u|^2/|x|^{N-2}dx$  as an integral over $\partial B_r$, and then estimate the later by Poincar\'e-Wirtinger inequality to get $f(r)\leq Cf'(r)$, where the constant $C$  is related to eigenvalue on the sphere.

Now if one tries to use the proof of Alt-Caffarelli-Friedmann \cite{acf} for the Neumann problem in $\R^N \setminus K$, while performing the integration by parts,  a  boundary term of type 
\begin{eqnarray}
\int_{K\cap B_r} u^2 \frac{\partial}{\partial \nu}\left(\frac{1}{|x|^{N-2}}\right) d\mathcal{H}^{N-1}
\end{eqnarray}
remains, which a priori prevents us to prove any monotonicity (here the integration over $K$ has to be understood as an integration on both sides of $K$ and $u$  may take a different value on each side). Of course this term does not occur  when $u=0$ on $K$, which happens for the Dirichlet problem.  

Now the key point in our proof relies on the following simple observation: if for instance $K$ is contained in a vectorial $N-1$-plane, then the normal vector on $K$ at point $x$ must be orthogonal to the vector $x$ itself. By consequence we have $\frac{\partial}{\partial \nu}(|x|^{2-N}) =(2-N)|x|^{-N}x\cdot \nu=0$ on $K$, and  the above extra term disappears. So follows the monotonicity, and it works the same way if $K$ is contained in a more general $N-1$ dimensional cone.

Moreover, as for Bonnet's monotonicity formula,  the  argument here is sharp too, in the sense that constant normalized energy happens if and only if $u$ is positively homogeneous of right degree.

Even if the proof of Theorem \ref{mainth} needs the monotonicity result for  half-planes only (and $N=3$), we will prove it in any dimension for more general cones  that we will call Neumann cones, and that we define as follows.

\begin{defin} \label{definNcone}If $N\geq 3$, a closed set $K\subset \R^N$ will be called a Neumann cone if the three following properties hold:
\begin{enumerate}
\item $K$ is a cone.
\item  $K\cap \mathbb{S}^{N-1}$ is $(N-2)$-rectifiable.
\item The embedding $W^{1,2}(\mathbb{S}^{N-1}\setminus K)\to L^2(\mathbb{S}^{N-1})$ is compact.
\end{enumerate}
If $N=2$ we keep the same definition but  replacing $(2)$ by 
\begin{enumerate}
\item[(2)']  $K\cap \mathbb{S}^{1}$ is a finite number of singletons.
\end{enumerate}
\end{defin}

  It is easy to check that a vectorial hyperplane in $\R^N$ is a Neumann cone, or that a half-hyperplane is a Neumann cone. We recall that for instance, the so-called cone property guarantees the compactness of   $W^{1,2}(\Omega)\to L^2(\Omega)$ (see \cite[Theorem 6.2. p.144]{adams}).  Moreover, when $K$ is a Neumann cone then the first positive eigenvalue of the Neumann Laplacian, denoted by  $\lambda_1(\mathbb{S}^2\setminus K)$, is well defined. This is also the case for any set $K$ contained in a Neumann cone. We refer to Section \ref{neumanncones} for the exact definition of rectifiability and basic properties of Neumann cones.

Now here is an example of statement for the monotonicity result (see Lemma \ref{monotonie} for the general case where assumption (2) of Lemma \ref{monotonie0}  is replaced by a more general topological assumption). We  call energy minimizer in $\R^N\setminus K$, a function $u$ that locally minimizes $\int |\nabla u|^2dx$ in $\R^N\setminus K$ (see Definition \ref{locdirmin}). This function is harmonic in $\R^N \setminus K$ with a Neumann boundary condition on $K$, i.e. zero normal derivative on $K$, in a weak sense.

\begin{lem} \label{monotonie0} Let $K \subset \R^N$ be a closed set satisfying the following assumptions.
\begin{enumerate}
\item $K$ is contained in a Neumann cone.
\item $\partial B_r\setminus K$ is connected for a.e. $r>0$.
\item $\exists \gamma>0$ s.t. for a.e. $r>0$, the first positive eigenvalue of the Neumann-Laplace-Beltrami operator on $\partial B_r\setminus K$ satisfies $\lambda_1(\partial B_r\setminus K)\geq \gamma/r^2$. 
\end{enumerate}
Then for every local energy minimizer $u$ in $\R^N \setminus K$ we have that
\begin{eqnarray}
\varphi:r\mapsto \frac{1}{r^{\alpha}} \int_{B_r}\frac{|\nabla u|^2}{|x|^{N-2}} \; dx \label{defphi}
\end{eqnarray}
is nondecreasing, where $\alpha$ is defined through
$$\alpha=\alpha(N,\gamma)=\sqrt{(N-2)^2+4\gamma}-(N-2).$$
Moreover, if $\varphi(r)$ is constant on an interval $[a,b]$ then for a.e. $r\in [a,b]$  the value $\gamma / r^2$ is the first positive eigenvalue for  the Neumann-Laplace-Beltrami operator on $\partial B_r \setminus K$ and the restritction of $u$ on $\partial B_r \setminus K$ is an associated eigenfunction.
\end{lem}

For instance when $N=3$, the eigenvalue $\gamma=3/4$ provides a monotonicity with power $\alpha=1$. The monotonicity is then employed to characterize blow-up and blow-in limits of a global minimizer when $K$ is contained in a half-plane,  and this is the key ingredient to prove Theorem \ref{mainth}.

The assumption of existing a point of $K$ on the edge of the half-plane is here the guarantee the existence of a point at which the normalized energy (i.e. the quantity denoted by $\varphi(r)$ in Lemma \ref{monotonie0} with $\alpha=1$) will not converge to $0$ when $r$ goes to zero. At this step of the proof we shall need  the regularity result  of \cite{lreg3D}. For an arbitrary closed $K$ contained in a half-plane, it is tempting to try to find a point on the edge of another suitable half-plane  by moving the first one. But notice that for instance the closed set $K:=\{(x,y,0) \in \R^3 ; y\geq e^{-|x|}\}$  is contained in the half-plane $\{z=0 \text{ and } y\geq -1\}$ but does not touch the edge of any  half-plane containing it, thus we cannot avoid our assumption using this kind of strategy.

It is well known that the natural power in the normalized energy of   a Mumford-Shah minimizer is indeed $\alpha=1$, because of the standard energy estimate $\int_{B_r}|\nabla u|^2dx \leq Cr^{N-1}$. To this respect, let us emphasis  that regarding to the exponent $\alpha$ in the monotonicity formula, the situation in dimension 3 is pretty much different than the situation in dimension 2. This is due to the fact that $\lambda_1(\mathbb{S}^2\setminus K)$  is decreasing with respect to increasing cracks on the sphere, at the contrary of what happens  in dimension 2,  where $\lambda_1(\mathbb{S}^1\setminus K)$  increases  while removing increasing arc of circles $K$ (see Remark  \ref{monotonieEi}  and \ref{monotonieEi2} for more details).

By consequence, in dimension 2, the more $K\cap \partial B(0,r)$ is large, better is the decay of energy, and you always have $\alpha\geq1$ as soon $K$ meets $\partial B(0,r)$ at least at one point (for instance when $K$ is connected and $r$ is small enough).  In dimension 3, it is not so easy  to obtain $\alpha\geq 1$, because to get  this  you need $\lambda_1\geq 3/4$ on the sphere, which is  not trivially verified in practice. It depends  on the geometry of the set $K\cap B(0,r)$, which can now be complicated even with zero $\mathcal{H}^{N-1}$-measure, whereas consisting essentially in a finite number of points in dimension 2.

But the situation is even worst than that:   if the value $3/4$ is achieved for some $K\cap \partial B(0,r)$ (for instance when $K$ is a half-plane), then any larger set $K'\supset K$ (with zero $\mathcal{H}^{N-1}$-measure on the sphere) would fail to have $\lambda_1\geq 3/4$, and consequently the larger set would imply a smaller exponent $\alpha$. This phenomenon is somehow counter-intuitive regarding to what happens in dimension 2.

All the above explains  why we have the restriction of $K$ being contained in a half-plane in the statement of Theorem \ref{mainth}. By this way we guarantee that assumption (3) of Lemma \ref{monotonie0} is satisfied with $\gamma=3/4$, thus the monotonicity holds with $\alpha=1$. If $K$ was larger than a half-plane then we would have $\gamma<3/4$ and the monotonicity with $\alpha<1$ would be useless regarding to blow-in or blow-down limits of a Mumford-Shah minimizer.

\section{Preliminaries}
\label{secprel}

We will denote by $\mathscr{L}^N$ the Lebesgue measure on $\R^N$. Sometimes we use the notation $|A|$ for $\mathscr{L}^N(A)$ and the symbol $dx$ in a integral will mean $d\mathscr{L}^N(x)$, and will be sometimes omitted if no confusion arizes. We denote   by $B(x,R)$ the open ball of center $x$ and radius $R$. The closed ball will be denoted by $\overline{B}(x,R)$, and when $x=0$ we will sometimes simply use the notation $B_R$ and $\overline{B}_R$. The topological boundary of a set $A\subset \R^N$ will be denoted by  by $\partial A$, and the unit sphere of $\R^N$ will be denoted either by $\partial B(0,1)$ or $\mathbb{S}^{N-1}$.

Let $E$ be a subset of $\R^N$. The $k$-dimensional Hausdorff measure of $E$ is
\begin{eqnarray}\label{defHH}
\Hh^k(E)=\lim_{\tau \to 0^+} \Hh^k_{\tau}(A), \notag
\end{eqnarray}
where
\begin{eqnarray}
 \Hh^k_{\tau}(A):=\inf \left\{ c_k\sum_{i=1}^{+\infty} \diam(E_i)^k \quad ; \quad E\subset
\bigcup_{i=1}^{+\infty}E_i \; \text{ and } \; \diam(E_i)\leq \tau\right\} ,\notag
\end{eqnarray}
and $c_k=\mathscr{L}^k(B(0,1))/2^k$. It is well known that   $\mathcal{H}^k$ is an outer measure on $\R^N$ for which the Borel sets are measurable sets. Moreover, its restriction to $k$-dimensional spaces coincides with the  Lebesgue measure $\mathscr{L}^k$ for $k\geq 1$, and $\mathcal{H}^0$ is the standard counting measure.

Sometimes we will need to integrate over spheres of $\R^N$, and this will mean endowed with the  Hausdorff measure $\mathcal{H}^{N-1}$ but  sometimes abbreviated by $d\omega$, or even omitted when there is no possible ambiguity.

\begin{defin}\label{locdirmin} Let $K\subset \R^N$ be a closed set and  $u \in \bigcap_{R>0} W^{1,2}(B_R\setminus K)$. We say that $u$ is a local energy minimizer in $\R^N\setminus K$ if for every $R>0$,  $u$ is a solution for the problem
$$\min \left\{ \int_{B_R\setminus K} |\nabla v|^2dx ; v\in W^{1,2}(B_R\setminus K) \text{ and } v=u \text{ on }\partial B_R\setminus K \right\}.$$
\end{defin}

\begin{rem} It is standard to check that $u$ will satisfy $-\Delta u =0$ in $B_R\setminus K$, and $\frac{\partial u}{\partial \nu}=0$ on $K$, in a weak sense. In particular $u$ is a $C^\infty$ function outside $K$. 
\end{rem}

We introduce the set of admissible pairs
$$\mathcal{A}:=\left\{(u,K) ; \; K\subseteq \R^N \text{ is closed and } u \in \bigcap_{R>0} W^{1,2}(B_R\setminus K)\right\}.$$

Following \cite{b} and \cite{d} we introduce the following definitions.

\begin{defin} Let   $(u,K) \in  \mathcal{A}$ and $B\subset \R^N$ be a ball. A competitor  for the pair  $(u,K)$ in the ball $B$ is a pair
$(v,L) \in \mathcal{A}$ such that
$$
\left.
\begin{array}{c}
u=v \\
K=L
\end{array}
\right\} \text{ in } \R^N \setminus  B
$$
and satisfying furthermore the following topological condition:  if $x$ and $y$ are two points in
$\R^N \setminus( B \cup K) $ that lie in different connected components of  $\R^N\setminus K$, then they also lie in different connected components of $\R^N\setminus L$.
\end{defin}

We can now define what is a global minimizer.

\begin{defin}\label{defms} A global minimizer in $\R^N$  is a pair $(u,K)\in \mathcal{A}$
such that for every ball $B$  in $\R^N$ and every
competitor $(v,L)$ in $B$ we have
$$\int_{B \backslash K}|\nabla u|^2dx +\mathcal{H}^{N-1}(K\cap B)\leq \int_{B\backslash L}|\nabla v|^2dx+\mathcal{H}^{N-1}(L\cap B).$$
The minimizer $(u,K)$ is called reduced if whenever  $(\tilde K, \tilde u)$ is a competitor in a ball $B$ with  $\tilde K\subset K$  and $\tilde u$ is an extension of $u$ in $W^{1,2}(B\setminus \tilde K)$ then $\tilde K=K$.
\end{defin}

\begin{rem} In particular, if $(u,K)$ is a  global minimizer in $\R^N$ then $u$ is a local energy minimizer in $\R^N\setminus K$.
\end{rem}

We end this section with a regularity result that will be needed later. We recall that a minimal cone  in $\R^3$ is a set belonging to the following list.
\begin{itemize}
\item A plane.
\item A union of three half-planes meeting along their edges with $120^o$.
\item  A cone over the edges of a regular tetrahedron.
\end{itemize}

We shall need the use of the following 3-dimensional ``$\varepsilon$-regularity'' result. There is no comparable theorem in higher dimensions (essentially because the minimal cones for $N\geq 4$ are unknown), and this is the  principal reason why our main result will be proved in dimension 3 only.

\begin{thm}\cite[Theorem 9]{lreg3D}  \label{regularity}We can find some absolute positive constants $\varepsilon_0>0$, $\alpha\in (0,1)$ and $c<1$ such that the following holds. Let $(u,K)$ be a reduced global minimizer in $\R^3$. Let $x\in K$ and $r>0$ be such that 
$$ \frac{1}{r^2}\int_{B(x,r)\setminus K}|\nabla u|^2 dx\leq \varepsilon_0.$$
Then there is a diffeomorphism $\phi$ of class $C^{1,\alpha}$ from $B(0,cr)$ to its image, verifying $\phi(0)=x$,   $|\phi(y)-(y+x)|\leq 10^{-3}cr$, and there is a
minimal cone $Z$ such that $K\cap B(x,cr)=\phi(Z)\cap B(x,cr)$.
\end{thm}

\section{Blow-up and Blow-in Limits of minimizers}

In this section we recall some known facts about  blow-up or blow-in limits of Mumford-Shah minimizers. Extracting sequences which converge to a minimizers in a fairly weak sense is not very difficult. The most difficult part is to show that this convergence holds strongly in $L^2_{loc}(\R^N)$ for the gradients.

If $K$ and $K'$ are two compact subsets of $\R^N$   the Hausdorff distance between $K$ and $K'$ is 
$$d_H(K,K')=\max\left(\sup_{x\in K} dist(x,K'), \sup_{x\in K'}dist(x,K)\right).$$
when a sequence of compact sets $K_n$ is said to converge to a set $K$, it will always refer to the convergence in the sense of the Hausdorff distance.

Subsequently,  the convergence of a sequence of minimizers $(u_k,K_k)$ will be understood in the following way.

\begin{defin} \label{convergence} We say that a sequence of couples $(u_k,K_k)\in \mathcal{A}$ converges to some $(u_0,K_0) \in \mathcal{A}$ if the following holds.
 \begin{enumerate}
\item $K_{k}\cap \overline{B}_M\to K_0\cap \overline{B}_M$ for every $M>0$. 
\item  for any connected component $U\subseteq \R^N \setminus K_0$, there exists a sequence of numbers $a_k$ such that  $u_{k}-a_k$ converges to $u_0$  strongly in $L^1(H)$, for every compact set $H\subset U$.
\item  $\nabla u_{k}\to \nabla u_0$  strongly in $L^2(B_M)$, for every $M>0$.
\end{enumerate}
\end{defin}

 The following statement summarizes the properties that we shall need later. It is mainly coming from gluing together several Propositions contained in the  book \cite{d},  applied to the special case of global minimizers.

\begin{thm} \label{blowup} Let $(u,K)$ be a  reduced global minimizer in $\R^N$ and for all $r>0$ let $(u_r,K_r)$ be the new global minimizer defined by 
$$u_r(x):=\frac{1}{\sqrt{r}}u(rx) \text{ and } K_r:=\frac{1}{r}K.$$
Then, there exists two sequences  $r_k\to 0$ and $R_k\to +\infty$ such that $(u_{r_k},K_{r_k})$ converges to some $(u_0,K_0)$ and $(u_{R_k},K_{R_k})$ converges to some $(u_\infty,K_\infty)$ in the sense of Definition \ref{convergence}. Moreover $(u_0,K_0)$ and $(u_\infty,K_\infty)$ are reduced global minimizers in $\R^N$.
\end{thm}

\begin{proof} The convergence of $K_r$ and $u_r$, together with a weak convergence for $\nabla u_r$ in $L^2(B_M)$ are  consequences of standard compactness results, as explained for instance in  \cite[Proposition D.37.8.]{d}. The starting point is the fact that $u_r$ have a uniform Dirichlet energy in $B_M$, due to the classical energy estimate for Mumford-Shah minimizers $\int_{B_r}|\nabla u|^2dx \leq Cr^{N-1}$, obtained by taking $(u{1}_{\R^N\setminus B_r}, (K\setminus B_r)\cup \partial B_r)$ as competitor.  This ``weak'' convergence is enough to obtain that $(u_0,K_0)$ and $(u_\infty,K_\infty)$ are again reduced global minimizers. This fact is quite not obvious but a proof is, for instance, given in   \cite[Theorem D.40.9]{d}. In addition, by \cite[Proposition D.37.18.]{d} (which is just semicontinuity with respect to the weak convergence), we get
$$\int_{B_M\setminus K_0} |\nabla u_0|^2 dx \leq \liminf_{k } \int_{B_M\setminus K_{r_k}} |\nabla u_{r_k}|^2 dx, \text{ and }$$
$$\int_{B_M\setminus K_\infty} |\nabla u_0|^2 dx \leq \liminf_{k } \int_{B_M\setminus K_{R_k}} |\nabla u_{R_k}|^2 dx.$$
Finally, the reverse inequality in the above with a limsup is again not obvious, but follows from \cite[Corollary D.38.48.]{d}, which together with the weak convergence implies strong convergence in $L^2(B_M)$ for the gradients. 
\end{proof}

\begin{rem} Actually, since $u_r$ is a sequence of harmonic functions, the convergence in compact sets of $\R^N \setminus K_0$ is even better: it is a uniform convergence \cite[Proposition D.37.25.]{d}, but we will not need it.
\end{rem}


\section{Geometric Properties of Neumann cones}
\label{neumanncones}

The purpose of this section is to give a rigorous formulation and justification of the fact that the normal vector  to a Neumann cone $K$ at some point $x\in K$, is  always orthogonal to the vector $x$ itself, which is one of the key ingredient to prove the monotonicity Lemma of Section \ref{monotonicity}.

If $E\subset \R^N$ and $1\leq k \leq N$, we say that $E$ is $k$-rectifiable if there exists at most countably many Lipschitz mappings $f_i:\R^k \to \R^N$ such that 
$$\mathcal{H}^{k}(E\setminus   \cup_i  f_i(\R^{k}))=0.$$

 We denote the Grassmannian by $G(N,k)$, i.e. the set of all vectorial subspaces of $\R^N$ of dimension $k$.

 Let $1\leq k\leq N$ and $K\subset \R^N$ be a set with locally finite $\mathcal{H}^{k}$ measure. We say that a plane $P \in G(N,k)$ is an approximate tangent plane of $K$ at point $x$  if
$$\int_{\frac{1}{r}(K-x)} \varphi(y) d\mathcal{H}^{k}(y)\underset{r\to 0}\longrightarrow \int_{P} \varphi(y) d\mathcal{H}^{k}(y) \quad \quad \forall \varphi \in C^{0}_c(\R^N).$$
We shall denote this plane by $T_x(K)$. If $K\subset \R^N$ is $k$-rectifiable  then it well known that it admits an approximate tangent plane at $\mathcal{H}^{k}$ a.e. point \cite[Theorem 2.83 (i)]{afp}.

The following proposition is an easy exercise, but we provide a detailed proof since it appears as a crucial fact regarding to our main result.

\begin{prop} \label{curtail}Assume $N\geq 3$. Let  $E\subset \mathbb{S}^{N-1}$ be $(N-2)$-rectifiable and  $C\subset \R^N$ be the cone generated by $E$, i.e. 
$$C:=\{t w \in \R^N\; ;\; t\geq 0 \text{ and } w \in E\}.$$
Then $C$ is $(N-1)$-rectifiable and for $\mathcal{H}^{N-1}$-a.e. $x \in C$, if  $T_x(C)\in G(N,N-1)$ is the approximate tangent plane to $C$ at point $x$ and $\nu_x$ is a  normal vector to $T_x(C)$, then  $x\cdot \nu_x =0$.
\end{prop}
\begin{proof}  Since the cone $\mathbb{S}^{N-1}\times \R^+$ is, away from the origin, locally bi-Lipschitz equivalent to an orthogonal cartesian product of type $\R^{N-1}\times \R$, it follows from \cite[Theorem 3.2.23]{federer} that $C$ is $(N-1)$-rectifiable in $\R^N$. By consequence $C$ admits an approximative tangent plane $T_x(C)\in G(N,N-1)$ at $\mathcal{H}^{N-1}$-a.e. point.

Let $E:=C\cap \mathbb{S}^{N-1}$. By assumption, $E$ is a $(N-2)$-rectifiable set. Let $S\subset E$ be the $\mathcal{H}^{N-2}$-negligible set such that $E\setminus S$ admits a $(N-2)$-dimensional approximate tangent plane   $T_w(E) \in G(N,N-2)$ for every $x\in E\setminus S$.

In virtue of \cite[Theorem 2.10.45]{federer}, the cone generated by $S$, namely $\bigcap_{t\geq 0} tS$ is $\mathcal{H}^{N-1}$-negligible. We deduce that for $\mathcal{H}^{N-1}$-a.e. $x\in C$, we can assume that both $T_x(C)$ exists and moreover $T_{x/|x|}(E)\in G(N,N-2)$  also  does exist. But then it is easy to check that $T_x(C)=T_{x/|x|}(E) \times vect\{x\}$.

Indeed, let  $x_0\in C$ be given and let us seek for the approximate tangent plane at point $x_0$. Since $C$ is invariant under dilatation, we can assume without loss of generality that $|x_0|=1$. For $r>0$ we set $C_r:=\frac{1}{r}(C-x_0)$. Then for every $\varphi \in C^0_c(\R^N)$, the coarea formula \cite[(2.72)]{afp} applied on the rectifiable set $C_r$ with the $r$-Lipschitz function $f:x\mapsto r|x+x_0|$,  for which under the notation of \cite{afp} the coarea factor is equal to ${\bf C}_1 d^{P} f_x = r$ for all $P \in G(N,N-1)$, yields
\begin{eqnarray}
\int_{C_r} \varphi(y) d\mathcal{H}^{N-1}(y)&=&\frac{1}{r} \int_{\R^+} \left(\int_{C_r\cap \partial B(-x_0,t/r)} \varphi(w) d\mathcal{H}^{N-2}(w)  \right) dt\notag \\
&=& \frac{1}{r}\int_{\R^+} \left(\int_{\frac{1}{r}\big(C\cap \partial B_t -x_0\big)} \varphi(w) d\mathcal{H}^{N-2}(w)  \right) dt \label{firmeou}
\end{eqnarray}
But since $C$ is conical, we have that $C\cap \partial B_t=tE$ for all $t$ and from the fact that $\mathcal{H}^{N-2}(tE)=t^{N-2}\mathcal{H}^{N-2}(E)$  we deduce that 
$$\int_{\frac{1}{r}\big(tE -x_0\big)} \frac{\varphi(w)}{r} d\mathcal{H}^{N-2}(w) =\int_{\frac{1}{r}\big(E -x_0\big)}  \frac{1}{r}\varphi\big(tz+x_0(t-1)/r\big) t^{N-2} d\mathcal{H}^{N-2}(z).$$
Let us denote by $\mu_{r}$ the $\sigma$-finite measure on $\R^N$ defined by $\mu_r:=\mathcal{H}^{N-2}|_{\frac{1}{r}(E-x_0)}$, that we know to converge weakly to $\mu_0:=\mathcal{H}^{N-2}|_{T_{x_0}(E)}$. Returning to \eqref{firmeou} and applying Fubini Theorem we arrive to 
\begin{eqnarray}
\int_{C_r} \varphi(y) d\mathcal{H}^{N-1}(y)= \int_{\R^N} \left( \int_{\R^+}  \frac{1}{r} \varphi\big(tz+x_0(t-1)/r\big) t^{N-1}  dt  \right) d\mu_r(z) \notag
\end{eqnarray}
or, using the change of variable $u=(t-1)/r$ in the integral on the $t$ variable,  we can also write it as
\begin{eqnarray}
\int_{C_r} \varphi(y) d\mathcal{H}^{N-1}(y)&=& \int_{\R^N} \left( \int_{-\frac{1}{r}}^{+\infty}  \varphi\big((ru+1)z+x_0 u\big) (ru+1)^{N-1} du  \right) d\mu_r(z) \notag \\
&=&  \int_{\R^N} \left( \int_{-\infty}^{+\infty}   \varphi\big((ru+1)z+x_0 u\big)(ru+1)^{N-1} du  \right) d\mu_r(z), \label{Fubinibini}
\end{eqnarray}
the latter being true for $r$ small enough, depending on the size of the support of $\varphi$, i.e. the integration being in reality computed on a compact domain only. Now we notice that, for $0\leq r\leq 1$ the family of continuous functions 
$$g_r : z\mapsto  \int_{-\infty}^{+\infty}   \varphi\big((ru+1)z+x_0 u\big) (ru+1)^{N-1} du  $$
are all supported on a same compact set, and   converges uniformly as $r\to 0$ to the function 
$$g_0(z):= \int_{\R}   \varphi\big(z+x_0 u\big) du.$$
Therefore, passing to the limit in \eqref{Fubinibini} we obtain that 
$$\int_{C_r} \varphi(y) d\mathcal{H}^{N-1}(y)\underset{r\to 0 }\longrightarrow \int_{\R^N} \left( \int_{\R}   \varphi\big(z+x_0 u)   du  \right) d\mu_0(y) = \int_{T_{x_0}(E)\times \R x_0}\varphi d\mathcal{H}^{N-1}$$
and this is exactly saying that $T_{x_0}(C)=T_{x_0}(E) \times vect\{x_0\}$, as desired.
 \end{proof}
 
 \begin{rem} A direct consequence of Proposition \ref{curtail} is that Neumann cones are $(N-1)$-rectifiable.
 \end{rem}

\begin{cor} \label{maincor}Let $N\geq 2$, $C\subset \R^N$ be a Neumann cone and $K\subset C$ a closed set. Then $K$ is $(N-1)$-rectifiable and  for $\mathcal{H}^{N-1}$a.e. $x \in K$ we have that $x \cdot \nu_x =0$ for any vector $\nu_x$ orthogonal to $T_x(C)$.
\end{cor}

\begin{proof} The statement is trivial for $N=2$ thus we assume that $N\geq 3$. From the fact that $K\subset C$ we directly get that $K$ is $(N-1)$-rectifiable (because $C$ is), and moreover $T_x(K)$ must be contained in $T_x(C)$. But since they have same dimension we must have $T_x(K)=T_x(C)$. The conclusion then follows from Proposition \ref{curtail}.
\end{proof}

\section{Integration by parts for  $SBV$ fields}
\label{IPPSBV}

Let us recall some standard definitions and properties of the functional spaces $BV$ and $SBV$ that will be used later, and that one can find for instance in  \cite{afp}.

\subsection{Short review on SBV theory}
For any open set $\Omega \subset \R^N$, the space $BV(\Omega)$ is the class of all functions $u \in L^1_{loc}(\Omega,\R)$ such that $Du$ (the derivative of $u$ in the distributional sense) is a finite measure. If $u \in (BV(\Omega))^N$ is a $BV$ vector field, a point $z \in \R^N$  is an approximate limit for $u$ at point $x$ if
$$\lim_{\rho \to 0} \frac{1}{|B(x,\rho)|}\int_{B(x,\rho)}|u(y)-z|dy=0.$$
The set $S_u$ of points where this property does not hold is called the \emph{approximate discontinuity set} of $u$, and the points $z$ for which the limit exists is called an approximative limit of $u$ at point $x$ and is denoted by $\tilde u(x)$. A remarkable result of  Federer and Vol'pert (see \cite[Th.
3.78.]{afp}) says that when $u \in (BV(\Omega))^N$, then $S_u$ is
$(N-1)$-rectifiable  and $D^su$ (the singular part of
$Du$ with respect to $\mathscr{L}^N$ in the Radon-Nikodym
decomposition $Du=D^au+D^su$) restricted to $S_u$ is absolutely
continuous with respect to $\Hh^{N-1}$. We will say that $u \in
(SBV(\Omega))^N$ when $D^s u$ is actually \emph{concentrated} on $S_u$.

The density of the regular part $D^au$ of $Du$ with respect to
$\mathscr{L}^N$, denoted by $\nabla u$, coincides
$\mathscr{L}^N$-a.e. with the \emph{approximate differential} of $u$
(see \cite[Th. 3.83.]{afp}). A function $u$ is approximately
differentiable at $x$ if there exists a  matrix $\nabla u(x)$ such
that
$$\lim_{\rho \to 0}\frac{1}{|B(x,\rho)|}\int_{B(x,\rho)}\frac{|u(y)-\tilde u(x)-\nabla u (x) . (y-x)|}{\rho}dy =0.$$

In the sequel we will also use the notion of \emph{trace} of $u$ on the singular set $S_u$.
Since $S_u$ is rectifiable, one can fix an orientation $\nu_u: S_u \to {\mathbb{S}^{N-1}}$ in such a way that for
$\Hh^{N-1}$-a.e. $x \in S_u$ the approximate tangent plane to $S_u$ at $x$ is orthogonal to the vector $\nu_u(x)$.
Then for any $x \in S_u$ and $\rho >0$ we define $B(x,\rho)^+:=B(x,\rho)\cap \{y ; \langle y,\nu_u(x) \rangle \geq 0\}$
and    $B(x,\rho)^-:=B(x,\rho)\cap \{y ; \langle y, \nu_u(x) \rangle \leq 0\}$. For $\Hh^{N-1}$-a.e. $x \in S_u$,
Theorem 3.77. of \cite{afp} provides the existence of traces $u^+(x)$ and $u^-(x)$ satisfying
\begin{eqnarray}
\lim_{\rho \to 0} \frac{1}{|B(x,\rho)^\pm|}\int_{B(x,\rho)^\pm}|u(y)-u^\pm(x)|dy=0. \label{traceSBV}
\end{eqnarray}
The set of points $x \in S_u$ where $u^\pm(x)$ exist is called the
\emph{jump set} and is denoted by $J_u$.  It can be shown
that $\Hh^{N-1}(S_u\setminus J_u)=0$ and for $x \in J_u$ the
quantity $(u^+(x)-u^-(x))$ is called the jump of $u$ at point $x$,
which sign depends on the orientation of $S_u$. Moreover for any $u
\in (SBV(\Omega))^N$ the representation
\begin{eqnarray}
Du=D^au+D^su = \nabla u \,\mathscr{L}^N  + (u^+-u^-)\otimes \nu_u \,\Hh^{N-1}|S_u, \label{Suu}
\end{eqnarray}
holds. 

Let us also mention that a simple approximation argument \cite[Proposition 4.4.]{afp} says that, if $K\subset \Omega$   is closed and $\mathcal{H}^{N-1}(K)<+\infty$, then any $u \in L^\infty(\Omega)\cap W^{1,1}(\Omega\setminus K)$ belongs to $SBV(\Omega)$ and $\mathcal{H}^{N-1}(S_u\setminus K)=0$. 

\subsection{Sets of finite perimeter} If $U \subset \R^N$ is open  we say that $U$ has finite perimeter if ${\bf 1}_U \in BV(\R^N)$. The reduced boundary  $\partial^*U$ is the set of points $x \in \R^N$ such that the limit $\lim_{r\to 0}D({\bf 1}_U)(B_r(x))/|B_r(x)|=:-\nu_{U}(x)$ exists and satisfies $|\nu_U(x)|=1$ (here we chosed the normal to be pointing outwise). In particular when $\mathcal{H}^{N-1}(\partial \Omega)<+\infty$, then ${\bf 1}_U \in SBV(\R^N)$ and  $\partial^*U$ is  equal to the jump set of ${\bf 1}_\Omega$. By consequence the following Gauss-Green formula holds
$$\int_{U}\div\, \varphi \, dx =  \int_{\partial^*U} \nabla \varphi \cdot \nu_U \, d\mathcal{H}^{N-1}(x) \quad \quad \forall \varphi \in (C^{1}_c(\R^N))^N.$$

\subsection{A Gauss-Green Formula for SBV fields}

\begin{lem} \label{lemIPP} Let $\Omega\subset \R^N$ be open  and  $F\in L^\infty\cap(SBV(\Omega))^N$. Let $\div \, F^a$ be the trace of the absolutely continuous part of $DF$ with respect to the Lebesgue measure and $S_F$ be the singular set of $F$.   Then for any open set  $U$ such that 
\begin{enumerate}
\item $\overline{U}\subset \Omega$
\item $\mathcal{H}^{N-1}(\partial U)<+\infty$   
\item $\mathcal{H}^{N-1}(S_F\cap \partial^* U)=0$
\end{enumerate}
 we have that
\begin{eqnarray}
 \int_U \div \,F^a \; dx = \int_{\partial^*U} \tilde{F} \cdot \nu_U \; d\mathcal{H}^{N-1} + \int_{S_F\cap U}  (F^- - F^+) \cdot \nu_F\; d\mathcal{H}^{N-1}, \label{IPPSBVF}
 \end{eqnarray}
where $\tilde{F}$ denotes the approximative limit of $F$, which exists $\mathcal{H}^{N-1}$-a.e. on the rectifiable set $\partial^*U$.
\end{lem}

\begin{proof} When $F \in (SBV(\Omega))^N$  we recall that 
\begin{eqnarray}
DF = \nabla F \,\mathscr{L}^N + (F^+ - F^-) \otimes \nu_F \,\mathcal{H}^{N-1}|_{J_F}. \label{Decompt}
\end{eqnarray}

By consequence $\div\,{F}$ is a measure and taking the trace in the above identity yields 
\begin{eqnarray}
\div \, F = \div \, F^a \,\mathscr{L}^N + (F^+ - F^-) \cdot \nu_F \,\mathcal{H}^{N-1}|_{J_F}. \label{Decompt1}
\end{eqnarray}
Now let  $U$ be such that $\overline{U}\subset \Omega$. It is well known that  $L^{\infty}\cap (SBV(\Omega))^N$ is an algebra, i.e. stable by products, as a consequence of the Chain-Rule in $BV$ (see \cite[Example 3.97.]{afp}). By consequence, $G={\bf 1}_{U} F \in (SBV(\Omega))^N$ and $S_G\subseteq S_F\cup \partial U$ (Remember that ${\bf 1}_U \in SBV$ because $\mathcal{H}^{N-1}(\partial U)<+\infty$). Since by assumption $\mathcal{H}^{N-1}(S_F\cap \partial U)=0$, we obtain that the decomposition of $\div \, G$ as a   measure is the following
\begin{eqnarray}
\quad \quad \div \, G= \div \, F^a \mathscr{L}^N|_{U} + (F^+ - F^-) \cdot \nu_F \,\mathcal{H}^{N-1}|_{J_F\cap U}- \tilde{F}\cdot \nu_U\,\mathcal{H}^{N-1}|_{\partial^*U}\label{identification}
\end{eqnarray}


(see again  \cite[Example 3.97.]{afp} for the characterization of $G^+$ and $G^-$ on the jump set in terms of $F^\pm$ and  ${\bf 1}_{U}^\pm$).

Now let us prove \eqref{IPPSBVF}. For this purpose we take a test function $\psi \in C^\infty_c(\Omega)$ such that $\psi=1$ on $\overline{U}$.  Then in the sense of distributions, $\langle \div \,G,\psi\rangle=\langle G, \div \,\psi\rangle= \int_{\Omega} G \,\div\,\psi\, dx=0$. But now returning to \eqref{identification} we get
\begin{eqnarray}
0&=&\langle \div \,G, \psi \rangle=\int_{\Omega}\psi \,d(\div \, G) \notag\\
&=&  \int_U \div \,F^a \; dx - \int_{\partial^*U} \tilde{F} \cdot \nu_U \; d\mathcal{H}^{N-1} + \int_{S_F\cap U}  (F^+ - F^-) \cdot \nu_F\; d\mathcal{H}^{N-1}. \notag
\end{eqnarray}
\end{proof}

\section{Spectrum of the Laplacian on a fractured sphere}

The purpose of this section is to study a Neumann eigenvalue problem on $\mathbb{S}^{N-1}\setminus K$ when $K$ is a Neumann cone, or more generally contained in a Neumann cone. 

We first assume that $K$ is a Neumann cone. We will call Neumann-Laplace-Beltrami operator on $\mathbb{S}^{N-1}\setminus K$, the operator associated to the quadratic form $Q(u)=\int_{\mathbb{S}^{N-1}\setminus K}|\nabla u|^2 \,d\mathcal{H}^{N-1}$ on the Hilbert space  $L^{2}(\mathbb{S}^{N-1})$,  with domain $W^{1,2}(\mathbb{S}^{N-1}\setminus K)$. Since by assumption the embedding $W^{1,2}(\mathbb{S}^{N-1}\setminus K)\to L^2(\mathbb{S}^{N-1})$ is compact, the quadratic form $Q(u)+\|u\|_2^2$ has a compact resolvent and it follows that the eigenvalues of the Neumann-Laplace-Beltrami operator  on $\mathbb{S}^{N-1}\setminus K$ is a countable and discrete set starting from $0$ and going to $+\infty$, denoted by $\lambda_k(\mathbb{S}^{N-1}\setminus K)$, $k \in \N$. Actually the eigenvalue problem can be studied separately in each connected component of $\mathbb{S}^{N-1}\setminus K$, in which the eigenvalue $\lambda_0(\mathbb{S}^{N-1}\setminus K)=0$ is simple. If $U$ is one connected component and $\gamma$ denotes the first positive eigenvalue in this domain, then the following Wirtinger inequality holds
\begin{eqnarray}
\quad \quad \int_{U} (u(x) -m)^2 d\mathcal{H}^{N-1}(x) \leq \frac{1}{\gamma}\int_{U} |\nabla_\tau u(x)|^2 d\mathcal{H}^{N-1}(x), \quad \forall u \in W^{1,2}(U),\label{wirtinger}
\end{eqnarray}
where $m$ denotes the average of $u$ on $U$ and $\nabla_\tau$ is the tangential gradient on the sphere.

Notice that if $C$ is a Neumann cone and if  $K$ is an arbitrary closed set such that $K\subset C$, then  the embedding $W^{1,2}(\mathbb{S}^{N-1}\setminus K)\to L^2(\mathbb{S}^{N-1})$ is still compact because $\|u\|_{W^{1,2}(\mathbb{S}^{N-1}\setminus C)}= \|u\|_{W^{1,2}(\mathbb{S}^{N-1}\setminus K)}$. By consequence all the above applies with $K$.

We will use several times the following simple fact.

\begin{rem}\label{monotonieEi} Let $E= \mathbb{S}^{2}\cap \{(x,y,z) \; ; \; z=0\}$ be the  equator, and let $F$ and $G$ be two subsets of $E$ verifying $F\subseteq G \subseteq E$. Then
$$\lambda_1(\mathbb{S}^2 \setminus G)\leq \lambda_1(\mathbb{S}^2 \setminus F).$$
This is a simple consequence of the Rayleigh formula, which says that 
$$\lambda_1(\mathbb{S}^2\setminus G)= \min_{v\in W^{1,2}(\mathbb{S}^2 \setminus G)}\frac{\int_{\mathbb{S}^2 \setminus G}|\nabla v|^2}{\int_{\mathbb{S}^2 \setminus G}v^2}.$$
The claim then follows from the fact that  $W^{1,2}(\mathbb{S}^2 \setminus F)\subseteq W^{1,2}(\mathbb{S}^2 \setminus G)$, thus the eigenfunction in $\mathbb{S}^2 \setminus F$ is a competitor in the minimum above, and its Rayleigh quotient in $\mathbb{S}^2 \setminus G$ is still equal to $\lambda_1(\mathbb{S}^2\setminus F)$ because  $G\setminus F$ has zero measure.
\end{rem}

\begin{rem} \label{monotonieEi2} Notice that the monotonicity described in Remark \ref{monotonieEi}  is the exact opposite of what happens in dimension 2 ! Indeed, if $I\subset \mathbb{S}^1$ is an arc of circle, then $\lambda_1(I)=(\pi/|I|)^2$, where $|I|$ is the length of the arc $I$. Therefore if $F\subset G$ are two arcs in $\mathbb{S}^1$, then $\lambda_1(\mathbb{S}^1 \setminus G)> \lambda_1(\mathbb{S}^1 \setminus F)$. What does not work in the above argument in dimension 2 is that the $L^2$ norm is not preserved in the respective two domains so that the Rayleigh quotient is not the same in the two domains, even if the inclusion of the spaces $W^{1,2}$ still holds.
\end{rem}

The eigenvalue in the special case of a half-plane is explicit. We sumarize some known facts in the following proposition.

\begin{prop} \label{valuehalf} Let 
$$S_\beta := \left\{(-\cos(\varphi),0,\sin(\varphi)) \in \mathbb{S}^{2};\; \, \varphi\in [-\beta,\beta]\right\}$$
be an arc of circle of aperture $2\beta$ on the unit sphere (vertical and left sided according to our parametrization). Then $\lambda_1(\mathbb{S}^{2}\setminus S_\frac{\pi}{2})=\frac{3}{4}$ is simple and the eigenspace  is generated by the restriction on $\mathbb{S}^{2}\setminus S_\frac{\pi}{2}$ of the following cracktip function written in cylindrical coordinates
\begin{eqnarray}
f(r,\theta,z)=\sqrt{r}\sin(\theta/2),\quad r>0 \text{ and }\theta\in(-\pi,\pi).\label{cracktipu}
\end{eqnarray}
In addition for $\beta$ close to $\frac{\pi}{2}$ we have the following expansion :
\begin{eqnarray}
\lambda_1(\mathbb{S}^{2}\setminus S_\beta)=\frac{3}{4}+\frac{2}{\pi}\cos(\beta)+o(\cos(\beta)). \label{developpement}
\end{eqnarray}
By consequence $\lambda_1(\mathbb{S}^{2}\setminus S_\beta) = 3/4$  if and only if  $\beta =\frac{\pi}{2}$ (i.e. $S_\beta$ is a half-equator).
\end{prop}
\begin{proof} The fact that $\lambda_1(\mathbb{S}^{2}\setminus S_\frac{\pi}{2})=\frac{3}{4}$ is simple and the eigenspace  is generated by the function in \eqref{cracktipu} is part of  \cite[Lemma 4.1.]{dauge}. Notice that a completely independent and somehow exotic proof can be found in \cite[Section 3]{lmg}.

The expansion \eqref{developpement} has been proved by Robert Legendre in \cite{legendre} (see alternatively \cite[page 53]{kmr} or \cite{brownS}). More precisely, Legendre found the expansion $\alpha=\frac{1}{2}+\frac{1}{\pi}\cos\theta+o(\cos\theta)$ for the first degree of homogeneity of an homogeneous Harmonic function with Neumann condition on a plane sector with aperture $2\theta$. Since the degree $\alpha$ is linked to the eigenvalue $\lambda_1$ through the relation
$$\alpha=\frac{1}{2}(\sqrt{1+4\lambda_1}-1),$$
we deduce \eqref{developpement} from Legendre's expansion.

Finally, the fact that $\lambda_1(\mathbb{S}^{2}\setminus S_\beta) = 3/4$  if and only if  $\beta =\frac{\pi}{2}$ follows by monotonicity of the eigenvalue with respect to the length of $S_\beta$ (i.e. Remark \ref{monotonieEi}), together with \eqref{developpement} which proves a strict monotonicity near the value $3/4$.
\end{proof}

Our main result is intimelly related to the knowledge of eigenvalues on the sphere. For instance the following rigidity result will be pivotal.

\begin{prop} \label{rigide}Let $E\subseteq \mathbb{S}^2$ be a half-equator and $K\subseteq E$ a closed set such that $\lambda_1(\mathbb{S}^2\setminus K)=3/4$. Then $K=E$.
\end{prop}
\begin{proof} We already know by Proposition \ref{valuehalf} that $\lambda_1(\mathbb{S}^2\setminus E)=3/4$. We also know that $\lambda_1(\mathbb{S}^2\setminus C)>3/4$ for any arc of circle $C \subseteq E$ with $C\not = E$.

Now assume that $K\subseteq E$ is closed and $K\not = E$. Even if all cases could be treated the same way, we shall give a separated and simpler argument in the case when, denoting by $d_E$ the geodesic distance in $E$,
$$\max_{x,y \in K} d_E(x,y)=\alpha<\pi.$$
If this occurs then there exists an arc of circle $C$ such that $\mathcal{H}^1(C)=\alpha<\pi$ and $K\subseteq C$. We deduce that $\lambda_1(\mathbb{S}^2\setminus K)\geq \lambda_1(\mathbb{S}^2\setminus C)>3/4$ (see Proposition \ref{valuehalf}), a contradiction.

Thus we are left with the case when
$$\max_{x,y \in K} d_E(x,y)=\pi.$$
Since $K$ is closed and $K \not = E$, there exists an open interval $I\subset E\setminus K$. Then $\lambda_1(\mathbb{S}^2\setminus K) \geq \lambda_1(\mathbb{S}^2\setminus (E\setminus I))$. Now we will prove that $\lambda_1(\mathbb{S}^2\setminus (E\setminus I)> 3/4$. Assume by contradiction that $\lambda_1(\mathbb{S}^2\setminus (E\setminus I)= 3/4$, and let $\varphi$ be an eigenfunction. Then $\varphi \in W^{1,2}(\mathbb{S}^2\setminus E)$, and 
$$\frac{3}{4}=\frac{\int_{\mathbb{S}^2\setminus E}|\nabla \varphi|^2}{\int_{\mathbb{S}^2\setminus E}\varphi^2}=\min_{v\in W^{1,2}(\mathbb{S}^2\setminus E)}\frac{\int_{\mathbb{S}^2\setminus E}|\nabla v|^2}{\int_{\mathbb{S}^2\setminus E}v^2}.$$
We deduce that $\varphi$ is an eigenfunction for the Neumann-Laplacian in $\mathbb{S}^2\setminus E$ as well, but since the corresponding eigenspace have dimension 1, $\varphi$ must be a multiple of the first eigenfunction in $\mathbb{S}^2\setminus E$, which is given in \eqref{cracktipu} (up to a suitable choice of cylindrical coordinates). We deduce  that $\varphi$ must have a jump on $I$, in other words does not belong to $W^{1,2}(\mathbb{S}^2\setminus (E\setminus I))$. This is a contraction.\end{proof}


\section{A monotonicity Lemma}
\label{monotonicity}

\begin{defin} Let  $K\subset \R^N$ be a closed set. We say that $K$ satisfies the topological condition (T) if the following holds. For almost every $r>0$ such that $\partial B_r\setminus  K$ is not connected, and for every connected component $U$ of $\partial B_r\setminus K$, there exists an open set $D\subset B_r\setminus K$ such that $\partial D \setminus U \subset K$. 
\end{defin}

\begin{rem} For instance if $\partial B_r\setminus  K$ is connected for a.e. $r>0$, then $K$ satisfies the topological condition (T). In particular this happens when $K$ is contained in a half-plane in $\R^3$. 
\end{rem}

The monotonicity Lemma relies on the following integration by parts formula. 
\begin{lem} \label{IPP} Let $K \subset \R^N$ be a closed set satisfying the following assumptions.
\begin{enumerate}
\item $K$ satisfies the topological condition (T).
\item $K$ is contained in a Neumann cone.
\end{enumerate}
Then for every local energy minimizer $u$ in $\R^N \setminus K$ (as in Definition \ref{locdirmin}) and for a.e. $r>0$ we have that
\begin{eqnarray}
\quad \quad \int_{B_r\setminus K} \frac{|\nabla u|^2}{|x|^{N-2}} dx \leq  \sum_{i\in I}r^{2-N}\int_{U_i} (u-m_i)\frac{\partial u}{\partial \nu}d\omega +\frac{N-2}{2r^{N-1}}\int_{U_i} (u-m_i)^2 d\omega ,\label{amontrer}
\end{eqnarray}
where $(U_i)_{i\in I}$ denotes the connected components of $\partial B_r\setminus K$ and $m_i$ is the average of $u$ on $U_i$.
\end{lem}
\begin{proof}This is just an integration by parts, but since the domain is not smooth, the full details are a bit technical.

We first regularize $|x|^{2-N}$ in the following way. For $\varepsilon>0$ we define 
$$|x|_{\varepsilon}:= \sqrt{x_1^2+x_2^2+\dots +x_N^2 +\varepsilon}$$
so that $|x|_{\varepsilon}$ is a $C^\infty$ function on $\R^N$.  A simple computation shows that
$$\Delta(|x|_{\varepsilon}^{2-N})=(2-N)N\frac{\varepsilon}{|x|_{\varepsilon}^{N+2}}\leq 0,$$ 
in other words $|x|_{\varepsilon}^{2-N}$ is superharmonic. Let us write, to lighten the notation, $h_\varepsilon:=|x|_{\varepsilon}^{2-N}$.

We will  successively  integrate two times  by parts. The first one can be treated by a variational argument, using the fact that $u$ is an energy minimizer. Indeed, let $\delta>0$ be a small parameter, and let $f(x)\in {\rm Lip}(\R^+ , \R)$ be the piecewise affine function such that  $f(x)=1$ for all $x \in [0,r-\delta]$, $f(x)=0$ for all $x\geq r$ and $f'(x)\leq \frac{1}{\delta}$ a.e. on $\R^+$.   Then let $\varphi_\delta(x):=f(|x|)$.   

Next, using that for any $t\in \R$, the function $u+t (\varphi_\delta u h_\varepsilon)$ is a competitor for $u$ in $W^{1,2}(B_r\setminus K)$, we deduce that 
$$\int_{B_r\setminus K} \nabla u \cdot \nabla (\varphi_\delta \,u \,h_\varepsilon)\,dx=0,$$
which implies 
$$\int_{B_r\setminus K} |\nabla u|^2  \, \varphi_\delta \, h_\varepsilon \,dx +\int_{B_r\setminus K} u \,h_\varepsilon\, \nabla u \cdot \nabla \varphi_\delta \,dx + \int_{B_r\setminus K}\varphi_\delta \,u\, \nabla u \cdot \nabla h_\varepsilon \,dx=0.$$
Letting $\delta \to 0$, we observe that $\varphi_\delta$ converges increasingly  to $1$ so that the first term in the above converges by the monotone convergence theorem. For the convergence of the second term we notice that $\nabla \varphi_\delta = -\frac{1}{\delta} \left(\frac{x}{|x|}\right)1_{B_r\setminus B_{r-\delta}}$ thus from Lebesgue's differentiation theorem we deduce that 
$$\int_{B_r\setminus K} u \,h_\varepsilon\, \nabla u \cdot \nabla \varphi_\delta \,dx \to_{\delta \to 0} -\int_{\partial B_r \setminus K}u \,h_\varepsilon\, \nabla u \cdot \nu,$$
for almost every $r>0$. Finally, the last term easily converges by the dominated convergence theorem so  in total we arrive to 
\begin{eqnarray}
\int_{B_r\setminus K} |\nabla u|^2   \, h_\varepsilon \,dx -\int_{\partial B_r\setminus K} u \,h_\varepsilon\, \nabla u \cdot \nu \,dx + \int_{B_r\setminus K} u\, \nabla u \cdot \nabla h_\varepsilon \,dx=0. \label{IPP1}
\end{eqnarray}

Now we want to apply a second integration by parts on the last term. Observe that we cannot argue as before because, even if $|x|^{2-N}$ is harmonic outside the origin, it is not an energy minimizer since it has infinite energy. So we need to integrate by parts for real, and for this we shall use the $SBV$ formula of Section \ref{IPPSBV}.

We first assume $u \in L^\infty$ and we remark that $F:=u^2\nabla h_\varepsilon \in W^{1,1}(B_r\setminus K)\cap L^\infty(B_r)$ and therefore belongs to $SBV(B_{r})$  due to \cite[Proposition 4.4.]{afp}.  In addition since $F$ is smooth outside $K$ we have that $S_F \subset K$ and $\mathcal{H}^{N-1}(K\cap \partial B_r)=0$ because it is contained in a Neumann cone.  By consequence, applying Lemma \ref{lemIPP} to $F$ with $U=B_r$ yields
\begin{eqnarray}
 \int_{B_r} \div F^a dx = \int_{\partial B_r } F \cdot \nu \; d\mathcal{H}^{N-1} +\int_{K\cap B_r}  (F^- - F^+) \cdot \nu_F\; d\mathcal{H}^{N-1}. \label{IPPF}
 \end{eqnarray}
Now let us identify each term: since $F$ is smooth outside $K$ it directly follows that 
$$\div F^a=2u\nabla u \cdot \nabla h_\varepsilon+ u^2 \Delta  h_\varepsilon \text{ a.e. in } B_r,$$
and 
$$F \cdot \nu = u^2 \frac{\partial h_\varepsilon}{\partial \nu} \quad \mathcal{H}^{N-1}\text{-a.e. on }\partial B_r.$$ 
Now we claim that 
$$ (F^- - F^+) \cdot \nu_F=0\quad, \quad \mathcal{H}^{N-1}\text{-a.e. on } K.$$
This is because $F^\pm=(u^2)^{\pm}\nabla h_\varepsilon$, $\mathcal{H}^{N-1}$- a.e. on  $K$ and $\nabla h_\varepsilon\cdot \nu_F=0$. Indeed, by definition, $\nu_F(x)$ must be orthogonal to the approximative tangent plane of $K$ at point $x$. But since $K$ is contained in a Neumann cone, Corollary \ref{maincor} says that $x\cdot \nu_F(x)=0$.   Therefore, \eqref{IPPF} simply becomes
\begin{eqnarray}
 \int_{B_r} 2u\nabla u \cdot \nabla h_\varepsilon+ u^2 \Delta  h_\varepsilon \;dx = \int_{\partial B_r } u^2 \frac{\partial h_\varepsilon}{\partial \nu} \; d\mathcal{H}^{N-1} . \label{IPPF2}
 \end{eqnarray}
This was assuming $u \in L^\infty$, but if not  we could replace $u$ by 
$$u^M:=\min(\max(u(x),-M),M),$$
 establish \eqref{IPPF2} for $u^M$ and then let $M\to +\infty$ to get $\eqref{IPPF2}$ for $u$ (notice that $u \in L^2$, $\nabla u \in L^2$, $\nabla h_\varepsilon \in L^\infty$ and that $|u^M|\leq |u|$, $|\nabla u^M|\leq |\nabla u|$ so that   $u^M \nabla u^M \cdot \nabla h_\varepsilon$ strongly converges in  $L^1$ to $u\nabla u \cdot \nabla h_\varepsilon$ as $M\to +\infty$ by the dominated convergence theorem. The other terms are treated by a similar way).

Gathering \eqref{IPPF2} together with \eqref{IPP1} we obtain that
 
\begin{eqnarray}
\int_{B_r\setminus K} |\nabla u|^2   \, h_\varepsilon \,dx = \int_{\partial B_r\setminus K} u \,h_\varepsilon\, \frac{\partial u}{\partial \nu}-  \frac{1}{2}\int_{\partial B_r \setminus K} u^2 \frac{\partial h_\varepsilon }{\partial \nu} +\int_{B_r\setminus K} u^2 \Delta  h_\varepsilon dx . \notag
\end{eqnarray}

and since $\Delta( |x|_\varepsilon^{2-N})\leq 0$,  we deduce that
\begin{eqnarray}
0\leq \int_{B_r\setminus K} |\nabla u|^2   \, h_\varepsilon \,dx \leq \int_{\partial B_r\setminus K} u \,h_\varepsilon\, \frac{\partial u}{\partial \nu}-  \frac{1}{2} \int_{\partial B_r \setminus K} u^2 \frac{\partial h_\varepsilon }{\partial \nu}  .\label{IPP4}
\end{eqnarray}
Now if we write the explicit expressions of $h_\varepsilon$ and $\frac{\partial h_\varepsilon }{\partial \nu}$ on the sphere $\partial B_r$ we get
$$0\leq \int_{B_r\setminus K} |\nabla u|^2 h_\varepsilon dx \leq (r^2+\varepsilon)^{\frac{2-N}{2}}\left(\int_{\partial B_r} u \frac{\partial u}{\partial \nu} \right) +(N-2)\frac{r}{2(r^2+\varepsilon)^{\frac{N}{2}}}\left(\int_{\partial B_r} u^2\right).$$ 
Finally, letting $\varepsilon \to 0$ we obtain
$$0\leq \int_{B_r\setminus K} \frac{|\nabla u|^2}{|x|^{N-2}}  \, dx \leq r^{2-N}\left(\int_{\partial B_r\setminus K} u \frac{\partial u}{\partial \nu} \right)+ \frac{N-2}{2r^{N-1}}\left(\int_{\partial B_r\setminus K} u^2\right).$$ 
Indeed notice that $\int_{B_r\setminus K} |\nabla u|^2 h_\varepsilon \, dx \to \int_{B_r\setminus K} \frac{|\nabla u|^2}{|x|^{N-2}}  \, dx$ by the monotone convergence theorem.

We are almost done, excepted the subtraction by the constants $m_i$. But this follows easily by applying all the above argument to the function $\tilde u := u- \sum_{i\in I} {\bf 1}_{D_i}(x)m_i$, where the $D_i$ are given by  the topological condition $(T)$. Notice that the domains $D_i$ are necessarily disjoint and that $\tilde u$ remains an  energy minimizer in $B_r\setminus K$ because we subtracted a locally constant function. 
\end{proof}

\begin{rem} Observe that,  from the fact that $u\in W^{1,2}(B \setminus K)$ for all $B$, the right hand side of the inequality  is bounded for a.e. $r>0$, thus a free consequence of Lemma \ref{IPP} is that $\int_{B_1} |\nabla u|^2 |x|^{2-N}dx <+\infty$, which is not obvious since a priori $\nabla u$ is only in $L^2(B(0,1))$. This says that $\nabla u$ is more integrable at the origin than expected, and  could be understood as a regularity result for energy minimizers outside a set $K$ contained in a Neumann cone.
\end{rem}

\begin{lem} \label{monotonie} Let $K \subset \R^N$ be a closed set satisfying the following assumptions.
\begin{enumerate}
\item $K$ satisfies the topological condition (T).
\item $K$ is contained in a Neumann cone.
\item $\exists \gamma>0$ s.t. for a.e. $r>0$ and for every connected component $U$ of $\partial B_r\setminus K$, the first positive eigenvalue of the Neumann-Laplace-Beltrami operator on $U$ satisfies $\lambda_1(U)\geq \gamma/r^2$.
\end{enumerate}
Then for every local energy minimizer $u$ in $\R^N \setminus K$ we have that
\begin{eqnarray}
\varphi:r\mapsto \frac{1}{r^{\alpha}} \int_{B_r}\frac{|\nabla u|^2}{|x|^{N-2}} \; dx \label{defphi}
\end{eqnarray}
is nondecreasing, where $\alpha$ is defined by
\begin{equation}
\alpha=\alpha(N,\gamma)=\sqrt{(N-2)^2+4\gamma}-(N-2). \notag
\end{equation}
Moreover, if $\varphi(r)$ is positively constant on an interval $[a,b]$ then for a.e. $r\in [a,b]$ and for every connected component $U$ of $\partial B_r\setminus K$,  the value $\gamma / r^2$ is the first positive eigenvalue for  the Neumann-Laplace-Beltrami operator on $U$ and the restritction of $u$ on $U$ is an associated eigenfunction.
\end{lem}

\begin{proof} We first apply Lemma \ref{IPP} and then use Cauchy-Schwarz and  the elementary inequality $ab\leq \frac{1}{2\delta}a^2 + \frac{\delta}{2}b^2$  to write
 \begin{eqnarray}
& & \int_{B_r\setminus K} \frac{|\nabla u|^2}{|x|^{N-2}}  \leq \sum_{i\in I}r^{2-N}\int_{U_i} (u-m_i)\frac{\partial u}{\partial \nu}d\mathcal{H}^{N-1} +\frac{N-2}{2r^{N-1}}\int_{U_i} (u-m_i)^2   d\mathcal{H}^{N-1}\notag\\
&\leq &  \sum_{i\in I}r^{2-N}\left(\int_{U_i} (u-m_i)^2\right)^{\frac{1}{2}}\left(\int_{U_i}\Big(\frac{\partial u}{\partial \nu}\Big)^2\right)^{\frac{1}{2}} +\frac{N-2}{2r^{N-1}}\int_{U_i} (u-m_i)^2  \notag\\
&\leq &  \sum_{i\in I}r^{2-N}\frac{r}{\sqrt{\lambda_1(U_i)}}\left(\int_{U_i} |\nabla_\tau u|^2\right)^{\frac{1}{2}}\left(\int_{U_i}\Big(\frac{\partial u}{\partial \nu}\Big)^2\right)^{\frac{1}{2}} +\frac{(N-2)r^2}{2r^{N-1}\lambda_1(U_i)}\int_{U_i} |\nabla_\tau u|^2   \notag\\
&\leq &  \sum_{i\in I}r^{2-N}\frac{r}{\sqrt{\gamma}}\left(\int_{U_i} |\nabla_\tau u|^2\right)^{\frac{1}{2}}\left(\int_{U_i}\Big(\frac{\partial u}{\partial \nu}\Big)^2\right)^{\frac{1}{2}} +\frac{(N-2)r^2}{2r^{N-1}\gamma}\int_{U_i} |\nabla_\tau u|^2   \notag\\
 &\leq& \sum_{i\in I}\frac{r^{3-N}}{\sqrt{\gamma}}\left(\frac{\delta}{2}\int_{U_i} |\nabla_\tau u|^2+\frac{1}{2\delta}\int_{U_i}\Big(\frac{\partial u}{\partial \nu}\Big)^2\right) +\frac{(N-2)r^{3-N}}{2\gamma}\int_{U_i} |\nabla_\tau u|^2   \notag\\
  &\leq& \sum_{i\in I}r^{3-N}\left(\frac{\delta}{2\sqrt{\gamma}}+ \frac{(N-2)}{2\gamma}\right)\int_{U_i} |\nabla_\tau u|^2+\frac{r^{3-N}}{2\delta\sqrt{\gamma}}\int_{U_i}\Big(\frac{\partial u}{\partial \nu}\Big)^2 .   \notag
\end{eqnarray}
Then we choose $\delta>0$ so that 
$$\left(\frac{\delta}{2\sqrt{\gamma}}+ \frac{(N-2)}{2\gamma}\right)=\frac{1}{2\delta\sqrt{\gamma}}.$$
This gives the value 
$$\delta=\frac{\sqrt{(N-2)^2+4\gamma}-(N-2)}{2\sqrt{\gamma}}$$
which implies the inequality,
\begin{eqnarray}
 \int_{B_r\setminus K} \frac{|\nabla u|^2}{|x|^{N-2}} \leq \frac{1}{\alpha} r^{3-N}\int_{\partial B_r\setminus K} |\nabla u|^2, \label{inequbound}
 \end{eqnarray}
with the $\alpha$ defined in the statement of the Lemma. But this exactly says that 
$$E(r)\leq \frac{1}{\alpha}r E'(r)$$
with 
$$E(r)=\int_{B_r\setminus K} \frac{|\nabla u|^2}{|x|^{N-2}},$$
in other words $\frac{d}{dr}(r^{-\alpha}E(r))\geq 0$ and $r^{-\alpha}E(r)$ is nondecreasing.

Finally to finish the proof, we observe that when $\varphi$ is constant, the derivative is zero, in other words  we must have $E(r)=\frac{1}{\alpha}rE'(r)$ for a.e. $r>0$, thus, since all the inequalities that we used to prove the monotonicity (starting from the second line) are nonnegative, or sum of nonnegative terms, they must all be equalities. In particular the equality in the third line says that the restriction of $u$ to almost every sphere must be the optimal function in the Poincar\'e-Wirtinger inequality associated with the constant $r^2/\gamma$, and so follows the Lemma.
\end{proof}

\begin{rem} Here for our convenience we considered energy minimizers, but the statement of Lemma \ref{monotonie} (or a simple variant) would still be true for stationary solutions, i.e. weak solution of a Neumann problem in all the $B(0,R)\setminus K$. Furthermore, arguing like in \cite{cl}, it would be possible to prove a similar monotonicity result for solutions of more general elliptic equations of the form $-\div A \nabla u = f$ for some $f\in L^p$ and H\"older regular coefficients $A(x)$.
\end{rem}


\section{Proof of main result}

\begin{proof}[Proof of Theorem \ref{mainth}] Without loss of generality, we can assume that the origin is situated on the edge of the half-plane, and contained in $K$. Thanks to Remark \ref{monotonieEi} and Proposition \ref{valuehalf}, we are exactly under the hypothesis of Lemma \ref{monotonie} with $\gamma=3/4$ because $K$ is contained in a half plane. Therefore  
$$\varphi(r):=\frac{1}{r}\int_{B_{r}\setminus K} \frac{|\nabla u|^2}{|x|^{N-2}}$$ 
is nondecreasing, and the limit in $0$ and $+\infty$ exists, let us denote them respectively by $f_0$ and $f_\infty$.  We claim that $f_0$ and $f_\infty$ are finite. To check this we shall use the following elementary inequality valid for any global Mumfors-Shah minimizer $(u,K)$ in $\R^N$,
$$\int_{B_r\setminus K}|\nabla u|^2 dx \leq \omega_N r^{N-1}\quad \forall r>0,$$
obtained simply by taking $(u{\bf 1}_{\R^N\setminus B_r},K\cup\partial B_r )$ as a competitor. Notice that the constant is  only dimensional: $\omega_N$ is the measure of the $N-1$-dimensional unit sphere.
Now using Fubini and Chebychev,  for any  $r>0$ we can chose $r_0 \in (r,2r)$ such that 
$$\int_{\partial B_{r_0}\setminus K}|\nabla u|^2 d\omega \leq \frac{1}{r} \int_{(B_{2r}\setminus B_r) \setminus K}|\nabla u|^2 dx \leq \omega_N(2r)^{N-2}.$$
From \eqref{inequbound} we deduce that
\begin{eqnarray}
\int_{B_{r}\setminus K} \frac{|\nabla u|^2}{|x|^{N-2}}\leq  \int_{B_{r_0}\setminus K} \frac{|\nabla u|^2}{|x|^{N-2}} \leq \frac{1}{\alpha} {r_0}^{3-N}\int_{\partial B_{r_0}\setminus K} |\nabla u|^2 \leq C(N,\alpha) r, \label{boundmagic}
\end{eqnarray}
with $C(N,\alpha)=\alpha^{-1}2^{N-1}\omega_N$. This implies that $\varphi(r)$ is uniformly bounded for  $r>0$ thus $f_0$ and $f_\infty$ are finite, and we readily have that $f_0\leq f_\infty<+\infty$.

Let us check  moreover that  $f_0>0$. Indeed  if not, then one would have, for $r>0$ small enough,
$$\frac{1}{r^{N-1}}\int_{B_r\setminus K}|\nabla u|^2 \leq \frac{1}{r}\int_{B_{r}\setminus K} \frac{|\nabla u|^2}{|x|^{N-2}}\leq \varepsilon_0,$$
where $\varepsilon_0$ is the same as the statement of Theorem \ref{regularity}. But then $K\cap B_{cr}$ must be the image of a minimal cone by a  $C^1$-smooth map, containing $0$ in its interior. This is a contradiction with the fact that the origin lies on the edge of the  half-plane containing $K$. We therefore conclude that $f_0>0$.

Now we take blow-up and blow-in limits. We begin with the blow-up. Let $u_k:=\frac{1}{\sqrt{r_k}}u(r_k x)$ and $K_k:=\frac{1}{r_k}K$. Then, by Theorem \ref{blowup}, up to a subsequence (not relabeled), as $r_k\to 0$ the sequence $(u_k,K_k)$ converges to some $(u_0,K_0)$ (in the sense of Definition \ref{convergence}), and $(u_0,K_0)$ is still a reduced global minimizer, with $K_0$ still satisfying assumptions $(1),(2),$ and $(3)$ of Theorem \ref{mainth}. Moreover for any $R>0$ and $k\in \N$ we have that 
\begin{eqnarray}
\frac{1}{R}\int_{B_{R}} \frac{|\nabla u_k|^2}{|x|^{N-2}}dx=\frac{1}{r_kR}\int_{B_{r_kR}}\frac{|\nabla u|^2}{|x|^{N-2}}dx \underset{k \to +\infty}\longrightarrow f_0. \label{densityy}
\end{eqnarray}
On the other hand, we know that $\nabla u_k$ converges to $\nabla u_0$ in $L^2(B_M)$ for any $M>0$. We actually claim that the following stronger convergence holds true
\begin{eqnarray}
\int_{B_{R}\setminus K} \frac{|\nabla u_k|^2}{|x|^{N-2}}dx \underset{k \to +\infty}\longrightarrow \int_{B_{R}\setminus K} \frac{|\nabla u_0|^2}{|x|^{N-2}}dx.
\end{eqnarray}
To prove the claim, we use  the estimate \eqref{boundmagic} and write
\begin{eqnarray}
\left|\int_{B_{R}} \frac{|\nabla u_k|^2}{|x|^{N-2}}dx-\int_{B_{R}} \frac{|\nabla u_0|^2}{|x|^{N-2}}dx\right| &\leq & \left|\int_{B_{R}\setminus B_\varepsilon} \frac{|\nabla u_k|^2}{|x|^{N-2}}dx-\int_{B_{R}\setminus B_\varepsilon} \frac{|\nabla u_0|^2}{|x|^{N-2}}dx\right| \notag\\
& & + \int_{B_{\varepsilon}} \frac{|\nabla u_k|^2}{|x|^{N-2}}dx +  \int_{B_{\varepsilon}} \frac{|\nabla u_0|^2}{|x|^{N-2}}dx\notag \\
&\leq &  \left|\int_{B_{R}\setminus B_\varepsilon} \frac{|\nabla u_k|^2}{|x|^{N-2}}dx-\int_{B_{R}\setminus B_\varepsilon} \frac{|\nabla u|^2}{|x|^{N-2}}dx\right| \notag \\
& & + 2C\varepsilon\notag
\end{eqnarray}
where $C=C(N,\alpha)$. Passing to the limsup and using the strong convergence of $|\nabla u_k|$ in $L^2$ we deduce that
$$
\limsup_{k\to +\infty}\left|\int_{B_{R}} \frac{|\nabla u_k|^2}{|x|^{N-2}}dx-\int_{B_{R}} \frac{|\nabla u_0|^2}{|x|^{N-2}}dx\right|  \leq 2C\varepsilon.
$$
The claim then follows letting $\varepsilon\to 0$.

Returning now to \eqref{densityy}, we obtain that 
$$\frac{1}{R}\int_{B_{R}} \frac{|\nabla u_0|^2}{|x|^{N-2}}dx=f_0, \quad \forall R>0.$$
But, since $f_0 \not = 0$,   the last conclusion of Lemma \ref{monotonie} then says that $\lambda_1(\partial B_R\setminus K_0)=3/4$ for a.e  $R>0$, and therefore  Lemma \ref{rigide} implies that, up to a $\mathcal{H}^{N-1}$ negligible set $G$,  $K_0$ is a half-plane. Since $K_0$ is reduced, then we claim that it must be the half-plane itself. Indeed, applying    Proposition 19 of \cite{lmg} to the couple $(K_0\cup G, u_0)$ we get that   $u$ must be  the function  $cracktip\times \R$, where cracktip is the function $\sqrt{2r/\pi}\sin(\theta/2)+C$ in suitable polar coordinates. But then the only reduced global minimizer in the equivalent class of couples $(K,u)$ for which $K$ coincides with the half plane a.e. and $u=cracktip\times \R$ a.e., is the whole half plane (because any ``hole'' in the half plane would create an irreversible jump for the function $u$ which would not be admissible in $W^{1,2}$).

Then we do exactly the same for $u_\infty$ and conclude that, akin to the blow-up limit, the normalized energy associated to the blow-down limit  must be constant  as well. Thus up to an additive constant,  $u_\infty$ is the same function $\sqrt{2r/\pi}\sin(\theta/2)$ as  $u_0$ (with  the same constant in front), and therefore  $f_0=f_\infty$. But then returning to the function $u$ and in virtue of the monotonicity of $\varphi(r)$, we deduce that  $\varphi(r)$ is constant for all $r>0$, and by consequence $K$ must be a half-plane itself.
\end{proof}

 \vspace{0.5cm}
{\bf Acknowledgments.} This work was supported by the ANR research project GEOMETRYA, ANR-12-BS01-0014-01.

\bibliographystyle{plain}

\def\cprime{$'$}


\begin{thebibliography}{}

\end{thebibliography}


\begin{thebibliography}{10}

\bibitem{adams}
Robert~A. Adams and John J.~F. Fournier.
\newblock {\em Sobolev spaces}, volume 140 of {\em Pure and Applied Mathematics
  (Amsterdam)}.
\newblock Elsevier/Academic Press, Amsterdam, second edition, 2003.

\bibitem{acf}
Hans~Wilhelm Alt, Luis~A. Caffarelli, and Avner Friedman.
\newblock Variational problems with two phases and their free boundaries.
\newblock {\em Trans. Amer. Math. Soc.}, 282(2):431--461, 1984.

\bibitem{afp}
Luigi Ambrosio, Nicola Fusco, and Diego Pallara.
\newblock {\em Functions of bounded variation and free discontinuity problems}.
\newblock Oxford Mathematical Monographs. The Clarendon Press Oxford University
  Press, New York, 2000.

\bibitem{b}
A.~Bonnet.
\newblock On the regularity of edges in image segmentation.
\newblock {\em Ann. Inst. H. Poincar\'e Anal. Non Lin\'eaire}, 13(4):485--528,
  1996.

\bibitem{brownS}
S.~Brown and K.~Stewartson.
\newblock Flow near the apex of a plane delta wing.
\newblock {\em J. Inst. Math. Appl.}, 5:206--216, 1969.

\bibitem{cl}
Antonin Chambolle and Antoine Lemenant.
\newblock The stress intensity factor for non-smooth fractures in antiplane
  elasticity.
\newblock {\em Calc. Var. Partial Differential Equations}, 47(3-4):589--610,
  2013.

\bibitem{dauge}
Monique Dauge.
\newblock Neumann and mixed problems on curvilinear polyhedra.
\newblock {\em Integral Equations Operator Theory}, 15(2):227--261, 1992.

\bibitem{d}
Guy David.
\newblock {\em Singular sets of minimizers for the {M}umford-{S}hah
  functional}, volume 233 of {\em Progress in Mathematics}.
\newblock Birkh\"auser Verlag, Basel, 2005.

\bibitem{dreg}
Guy David.
\newblock {$C^{1+\alpha}$}-regularity for two-dimensional almost-minimal sets
  in {$\Bbb R^n$}.
\newblock {\em J. Geom. Anal.}, 20(4):837--954, 2010.

\bibitem{dl}
Guy David and Jean-Christophe L{\'e}ger.
\newblock Monotonicity and separation for the {M}umford-{S}hah problem.
\newblock {\em Ann. Inst. H. Poincar\'e Anal. Non Lin\'eaire}, 19(5):631--682,
  2002.

\bibitem{federer}
Herbert Federer.
\newblock {\em Geometric measure theory}.
\newblock Die Grundlehren der mathematischen Wissenschaften, Band 153.
  Springer-Verlag New York Inc., New York, 1969.

\bibitem{kmr}
V.~A. Kozlov, V.~G. Maz{\cprime}ya, and J.~Rossmann.
\newblock {\em Spectral problems associated with corner singularities of
  solutions to elliptic equations}, volume~85 of {\em Mathematical Surveys and
  Monographs}.
\newblock American Mathematical Society, Providence, RI, 2001.

\bibitem{legendre}
Robert Legendre.
\newblock \'{E}coulement subsonique transversal \`a un secteur angulaire plan.
\newblock {\em C. r. hebd. S\'eanc. Acad. Sci.}, 243:1716--1718, 1956.

\bibitem{l}
J.~C. L\'{e}ger.
\newblock Flatness and finiteness in the {M}umford-{S}hah problem.
\newblock {\em J. Math. Pures Appl. (9)}, 78(4):431--459, 1999.

\bibitem{lmg}
Antoine Lemenant.
\newblock On the homogeneity of global minimizers for the {M}umford-{S}hah
  functional when {$K$} is a smooth cone.
\newblock {\em Rend. Semin. Mat. Univ. Padova}, 122:129--159, 2009.

\bibitem{lreg3D}
Antoine Lemenant.
\newblock Regularity of the singular set for {M}umford-{S}hah minimizers in
  {$\Bbb R^3$} near a minimal cone.
\newblock {\em Ann. Sc. Norm. Super. Pisa Cl. Sci. (5)}, 10(3):561--609, 2011.

\bibitem{lms}
Antoine Lemenant, Emmanouil Milakis, and Laura~V. Spinolo.
\newblock Spectral stability estimates for the {D}irichlet and {N}eumann
  {L}aplacian in rough domains.
\newblock {\em J. Funct. Anal.}, 264(9):2097--2135, 2013.

\bibitem{merlet}
Beno{\^{\i}}t Merlet.
\newblock Numerical study of a new global minimizer for the {M}umford-{S}hah
  functional in {$\bold R^3$}.
\newblock {\em ESAIM Control Optim. Calc. Var.}, 13(3):553--569, 2007.

\end{thebibliography}

\end{document}